\documentclass[a4paper,reqno]{amsart}

\usepackage{amsmath}
\usepackage{paralist}
\usepackage{graphics}
\usepackage{epsfig}
\usepackage{amsfonts}
\usepackage{amssymb}
\usepackage{hyperref,cite}
\usepackage{xcolor}

\UseRawInputEncoding
\usepackage{amsthm}
\usepackage{enumerate}
\usepackage[mathscr]{eucal}
\usepackage{eqlist}
\usepackage[misc]{ifsym}

\def\mobi{{M{\"{o}}bius}}

\newtheorem{theorem}{Theorem}[section]
\newtheorem{lemma}[theorem]{Lemma}

\def\ifl{\iffalse }

\numberwithin{equation}{section}

\newtheorem{corollary}[theorem]{Corollary}
\numberwithin{equation}{section}

\theoremstyle{remark}

\begin{document}

\baselineskip=1.2\baselineskip
\title[ On the dimension distortions of quasi-symmetric  homeomorphisms]
{ On the dimension distortions of quasi-symmetric  homeomorphisms}

\author{Shengjin Huo
}
\address{
Department of Mathematics, Tiangong University, Tianjin 300387, China} \email{huoshengjin@tiangong.edu.cn}




\begin{abstract}
 In this paper, we first  generalize a result of Bishop and Steger [Representation theoretic rigidity in PSL(2, R). Acta Math., 170, (1993), 121-149] by proving that for a Fuchsian group $G$ of divergence type and non-lattice, if $h$ is a quasi-symmetric homeomorphism of the real axis $\mathbb{R}$ corresponding to a quasi-conformal compact deformation of $G$. Then for any $E\subset \mathbb{R}$, we have  max(dim$E$,  dim$h(\mathbb{R}\setminus E))=1$. Furthermore, we showed that Bishop and steger's result does not hold for the covering groups of all '$d$-dimensional jungle gym' (d is any positive integer) which generalizes G\"onye's results [ Differentiability of quasi-conformal maps on the jungle gym. Trans. Amer. Math. Soc. Vol 359 (2007), 9-32] where the author discussed the case of '$1$-dimensional jungle gym'.
\end{abstract}

\maketitle

\section{ Introduction }

Let $G$ be a non-elementary torsion free discrete {\mobi} transformations group  acting on $\bar{\mathbb{R}}^{n}=\mathbb{R}^{n}\cup {\infty}$ or $S^{n}=\partial \mathbb{B}^{n}$; the action of $G$ can extend to the $(n+1$)-dimension  hyperbolic upper half hyperplane $\mathbb{H}^{n+1}=\{(x_{1},\cdot\cdot\cdot, x_{n+1})\in \mathbb{R}^{n+1}:x_{n+1}>0\}$ or the hyperbolic unit ball $\mathbb{B}^{n+1}.$  A discrete group $G$ is called  a Kleinian group if $n=2$ and Fuchsian group if $n=1.$  In this paper, we mainly focus our attention on  Fuchsian groups.

Let $\Lambda(G)$ be the accumulation set of any orbit. A Fuchsian  group $G$ is said to be of the first kind if  the limit set $\Lambda(G)$ is the entire circle. Otherwise, it is of the second kind. A point $x\in \bar{\mathbb{R}}$ is a conical point or radial point of $G$ if there is a sequence of elements $g_{i}\in G$ such that for any $z\in \mathbb{H}$, there exists a constant $C$ and a hyperbolic line $L$ with endpoint $x$ such that the hyperbolic distance between $g_{i}(z)$ and $L$ are bounded by $C.$ Denote by $\Lambda_{c}(G)$ the set of all the conical limit points and $\Lambda_{e}(G)$ the set of all the escaping limit points. Let $S=\mathbb{H}/G$ be the corresponding surface of $G$. The points in $\Lambda_{c}(G)$ are just corresponding to the geodesics in $S$ which return to some compact set infinitely often and the points in $\Lambda_{e}(G)$ are corresponding to the geodesics in $S$ which eventually leave every compact subset of $S$.

For $g$ in $G$, we denote by $\mathcal{D}_{z}(g)$ the closed hyperbolic half-plane containing $z$, bounded by the perpendicular bisector of the segment $[z, g(z)]_{h}$. The Dirichlet fundamental domain $\mathcal{F}_{z}(G)$ of $G$ centered at $z$ is the intersection of all the sets $\mathcal{D}_{z}(g)$ with $g$ in $G-\{id\}$. For simplicity, in this paper we use the notation $\mathcal{F}$ for the Dirichlet fundamental domain $\mathcal{F}_{z}(G)$ of $G$ centered at $z=0.$ A Fuchsian group $\Gamma$ is called a lattice if the area of its one Dirichlet fundamental domain is finite. Moreover, a lattice is said to be uniform if each of its Dirichlet domain is compact, for more details, see \cite{Da}.

 A Fuchsian  group $G$ is said to be of divergence type if $\Sigma_{g\in G}(1-|g(0)|)=\infty$. Otherwise, we say it is of convergence type. All the second kind groups are of convergence type but the converse is not true.

 We call $F$ a quasi-conformal deformation of $G$ if it is a quasi-conformal homeomorphism  of the upper half plane $\mathbb{H}$ such that $$G'=\{g': g'=F\circ g\circ F^{-1} \quad \text {for every } g\in G\}$$ is also a Fuchsian group and a compact quasi-conformal deformation of $G$ if it is just a lifted mapping  of a quasi-conformal mapping $f$ defined on the surface $\mathbb{H}/G$ whose Beltrami coefficients is supported on a compact subset of  $\mathbb{H}/G$. Such a $F$ will extend unique to a homeomorphism of the real axis $\bar{\mathbb{R}}$, denoted by $h$. The homeomorphism $h$ is a quasi-symmetric mapping of $\bar{\mathbb{R}}$.

The quasi-symmetric mappings can be very singular in the measure theoretic sense. It is known that quasi-conformal mapping preserve the null-sets. However, the quasi-symmetric mappings may be very singular, which will not preserve null-sets, see \cite{BA}.

In \cite{Tu}, Tukia showed that, for the unit interval $I=[0,1]$.  there are a quasi-symmetric self mapping of $I$ and  a set $E\subset I$ such that the Hausdorff dimensions of both $I\setminus E$ and $f(E)$ are less than 1. In \cite{BS}, Bishop and Steger got the following result: for a lattice group $G$ (i.e. G is finitely generated of first kind), there is a set $E\subset \mathbb{R}$ such that the hausdorff dimensions of both $E$ and $h(\mathbb{R}\setminus E)$ are less than 1, where $h$ is a quasi-symmetric conjugating homeomorphism of the real axis $\mathbb{R}.$

 Concerning the negative results we first give the definition of the 'd-dimensional jungle gym'.  Let $S_{0}$ be a compact surface of genus $d$ and $G_{0}$ its covering group. Let $N_{0}$ be a normal subgroup of $G_{0}$ such that $G_{0}/N_{0}$  is
isomorphic to $\mathbb{Z}^{d}$. The surface $S^{*}=\mathbb{H}/N_{0}$ is the so called infinite 'd-dimensional jungle
gym', that is, $S^{*}=\mathbb{H}/N_{0}$ can be quasi-isometrically embedded into $\mathbb{R}^{d}$ as a surface $S$
which is invariant under translations $t_{j}$, $1\leq j\leq d$, in $d$ orthogonal directions.
Moreover $S_{0}$ is conformal equivalent to $S/<t_{1}, \cdot\cdot\cdot, t_{d}>$.

In \cite{Go}, G\"{o}nye showed that Tukia-Bishop-Steger's results do not hold for the covering group of '1-dimensional jungle gym.' Gonye constructed a conjugating map $f$ between covering groups of two '1-dimensional jungle gym' with the Beltrami coefficient being compactly supported, for which
$$\text{max}(dim(E), dim f(\mathbb{R}\setminus E))=1$$
for all $E\subset R$.

In this paper we continue to investigate the range of validity of Tukia-Bishop-Steger results. We first show the following result which is essentially due to Fernandez and  Melian \cite{FM}.
\begin{theorem}\label{main1}
 Suppose $G$ is a non-lattice divergence type Fuchsian group. Then $\Lambda_{e}(G)$ has zero 1-dimensional Hausdorff measure, but its Hasudorff dimension is 1.
\end{theorem}

In \cite{Bi}, Bishop showed that the divergence Fuchsian groups have  Mostow rigidity property, so if $h$ is any quasi-symmetric homeomorphism which conjugates a divergence Fuchsian group to another one, then $h$ is singular, i.e. $h$ is continuous but the derivation of $h$ vanishes almost everywhere in the real axis $\mathbb{R}$. For the  quasi-symmetric homeomorphisms corresponding to a compact deformation of a divergence Fuchsian group, we have

\begin{theorem}\label{main2}
 Let $G$ be a  Fuchsian group of divergence type and non-lattice, and $h$ be  a homeomorphism of the real axis $\mathbb{R}$ corresponding to  a compact deformation of $G$.  Then for any $E\subset R,$ we have
$$\text{max}(\text {dim }(E), \text {dim }h(\mathbb{R}\setminus E))=1.$$
\end{theorem}

Combine with Bishop and Steger's result \cite{BS}, we have

\begin{theorem}\label{main3}
Let $G$ be a Fuchsian group and $h$  a quasi-symmetric homeomorphism of  the real axis $\mathbb{R}$ corresponding to  a compact deformation of $G$. Then  there exists a subset $E\subset \mathbb{R},$ such that
$$\text{max}(\text {dim }(E), \text { dim } h(\mathbb{R}\setminus E))<1\eqno(1.1)$$
if and only if $G$ is a lattice.
  \end{theorem}

Concerning the 'jungle gym',  by Theorem \ref{main3}, we can generalize G\"{o}nye's result to '$d$-dimensional jungle gym', where $d$ is any positive integer number.

\begin{corollary}\label{main4}
For any positive integer number $d$, suppose $G$ be a covering group of a '$d$-dimensional jungle gym' and $h$  a homeomorphism of  the real axis $\mathbb{R}$ corresponding to  a compact deformation of $G$. Then  for any $E\subset \mathbb{R},$ we have
$$\text{max}(dim(E), dim h(\mathbb{R}\setminus E))=1. \eqno(1.2)$$
\end{corollary}

{\noindent {\bf Remark:} } By \cite{AZ} we know that when $d = 1$ or $2$; the
covering  groups  of '$d$-dimensional jungle gyms' are of the divergence type and when $d\geq 3$, the covering  group of '$d$-dimensional jungle gyms' are of  the convergence type.

The remainder of the paper is organized as follows: In section 2 we recall some definitions. In section 3, we give some results about differentiability of Quasi-conformal mappings at escaping limit points. In section 4, we prove Thorem \ref{main1}. In section 5, we prove Thorem \ref{main2} and in section 6, we prove Thorem \ref{main3}.
\section{Preliminaries}

Before give the proofs of the above results, we first recall some definitions.

\noindent{\bf  2.1 Quasi-conformal mapping.} Let $\mathbb{H}$ be the upper half-plane in the complex plane
$\mathbb{C}$. We denote by $M(\mathbb{H})$  the unit sphere of the space $L^{\infty}(\mathbb{H})$ of all essentially bounded Lebesgue  measurable  functions in  $\mathbb{H}$. For a given $\mu \in
M(\mathbb{H})$, there exists a unique quasiconformal self-mapping
$f^{\mu}$ of $\mathbb{H}$ fixing $0$, $1$ and $\infty$, and satisfying the following equation
$$\frac{\partial}{\partial \bar{z}}f^{\mu}(z)=\mu(z)\frac{\partial}{\partial z} f^{\mu}(z),~~~~~a.e. ~z\in \mathbb{H}.$$
We call $\mu$  the Beltrami coefficient of $f^{\mu}.$ It is well known that $f^{\mu}$ can be extended continuously to the real axis $\mathbb{R}$ such that $f^{\mu}$ restricted to $\mathbb{R}$ is a quasisymmetric homeomorphism.

Similarly, there exists a unique quasiconformal homeomorphism
$f_{\mu}$ of the plane $\mathbb{C}$ which is holomorphic in the lower half-plane, fixing $0$, $1$ and $\infty$ and satisfying

$$\frac{\partial}{\partial \bar{z}}f_{\mu}(z)=\mu(z)\frac{\partial}{\partial z} f_{\mu}(z), ~~~~~a.e.~~ z\in \mathbb{H}.$$

\noindent{\bf 2.2 Poincar\'e exponent.}   The critical exponent (or Poincar\'e exponent) of a Fuchsian group $G$ is defined as
\begin{align}\delta(G)&=\inf\{t: \sum_{g\in G}\exp(-t\rho(0,g(0)))<\infty\}\\
&=\inf\{t:\,\sum_{g\in G}(1-\vert g(0)\vert)^t<+\infty\},
\end{align}
where $\rho$ denotes the hyperbolic metric. It has been proven in\cite {BJ}  that for any non-elementary group $G$, $\delta(G)$ is equals to $dim(\Lambda_c(G))$,  the Hausdorff dimension of the conical limit set.

\noindent{\bf 2.2 Hausdorff dimension.}  Let $E$ be a subset of the complex plane $\mathbb{C}$. Suppose $\varphi$ is an nonnegative increasing homeomorphism of $[0,\infty)$.
For $\varphi$ and
$0<\delta\leq \infty$, we  define
$$\mathcal{H}_{\delta}^{\varphi}(E)=inf\{\sum\limits_{i = 1}^\infty  {\varphi(|B_{i}|)} :E \subset \bigcup\limits_{i = 1}^\infty  {{B_i}} , |B_{i}|\leq\delta\},$$
 where $B_{i}\subset \mathbb{C}$ is a set and
$|B_{i}|$ denotes its diameter, the infimum is taken over all open coverings of $E$.  Then  the Hausdorff
measure of $E$ to be

$$\mathcal{H}^{\varphi}(E)=\mathop {\lim }\limits_{\delta  \to 0} \mathcal{H}_{\delta}^{\varphi}(E)=\mathop {\sup }\limits_{\delta  > 0} \mathcal{H}_{\delta}^{\varphi}(E)$$ and the Hausdorff content of $E$ is $\mathcal{H}_{\infty}^{\varphi}(E)$

If $\varphi(t)=t^{\alpha}$, $\alpha\in [0,2] $, we denote $\mathcal{H}^{\varphi}(E)$ by $\mathcal{H}^{\alpha}(E).$
Then one defines the $\alpha$-dimensional Hausdorff
measure of $E$ to be

$$\mathcal{H}^{\alpha}(E)=\mathop {\lim }\limits_{\delta  \to 0} \mathcal{H}_{\delta}^{\alpha}(E)=\mathop {\sup }\limits_{\delta  > 0} \mathcal{H}_{\delta}^{\alpha}(E).$$ One defines the Hausdorff dimension
of $E$ to be

$$dim E=inf \{\alpha:\mathcal{H}^{\alpha}(E)=0\}.$$

\section{Differentiability of Quasi-conformal mappings at escaping limit points revisited }

It is well known that a quasi-conformal mapping of a domian $\Omega$ is differentiable almost everywhere in $\Omega$. In this paper we need the following criterion for pointwise conformality due to Lehto , see  \cite{Le} or (\cite{LV}, Theorem 6.1.).
\begin{lemma}\label{le1}
Let $\Omega$ and $\Omega'$ be two domains in the complex plane $\mathbb{C}$, and let $f$ be a quasi-conformal mapping from  $\Omega$ to $\Omega'$ with Beltrami coefficient $\mu(z)$, where $|\mu(z)|\leq k<1$ almost everywhere in $\Omega.$ If $f$ satisfies
$$\displaystyle\frac{1}{2\pi}\iint_{|z|<r}\frac{|\mu(z)|}{|z|^{2}}dxdy<\infty \eqno(3.1) $$
for some $r>0$, then $f$ is conformal at $z=0.$
\end{lemma}

Suppose $G$ be a non-lattice Fuchsian group of divergence type. Let $f$ be a quasi-conformal mapping on the surface $S=\mathbb{H}/G$ whose Beltrami coefficient $\mu$ is supported on a compact subset of $S$. Thus we can choose a point $z_{0}\in S $ and a sufficiently large $r_{0}$ such that the support set of $\mu$ is contained in the disk  $B(z_{0},r_{0}).$ Let $S_{r_{0}}=S\setminus B(z_{0}, r_{0})$ and let $\Omega_{r_{0}}$ be the lift of $S_{r_{0}}$ to the upper plane $\mathbb{H}.$ By the definition of escaping limit points we know that an escaping geodesic eventually stays inside in the region $\Omega_{r_{0}}$ and far from the support of $\mu.$  We can lift $f$ to the upper half plane $\mathbb{H}$ and  get a quasi-conformal homeomorphism $F^{\mu}(z): ~~\mathbb{H}\rightarrow \mathbb{H}.$ Induced by the Beltrami coefficient of $F^{\mu}$ we can get a quasi-conformal homeomorphism of the complex plane $\mathbb{C}$ such that the Beltrami coefficient of $F_{\mu}$ is almost every equal to the one of $F^{\mu}$ on the upper half plane $\mathbb{H}$ and vanishes almost every on the lower half plane $\mathbb{L}.$

 For the quasi-conformal mapping $F_{\mu}$, we have following results which is similar to (\cite{Go }, Theorem 1.1) where G\"onye  discussed the differentiability of quasi-symmetric homeomorphism conjugating the covering groups of 1-dimensional 'Jungle Gym'.

\begin{theorem}\label{diff}
Suppose $G$ be a  Fuchsian group of divergence type and non-lattice,  and let $f$ be a quasi-conformal mapping on the surface $S=\mathbb{H}/G$ so that the  Beltrami coefficient $\mu$ of $f$ is  compactly supported on $S$. Let $F^{\mu}$ be the lifted mapping of $f$ to the upper half plane $\mathbb{H}$  extended to the real axis $\mathbb{R}$. Then $F^{\mu}$ is differentiable at the escaping points $x\in\Lambda_{e}(G)$ with the Jacobian $J(F^{\mu})=|(F^{\mu})'(x)|^{2}$. Furthermore, if $F_{\mu}$ be a quasi-conformal homeomorphism of the complex $\mathbb{C}$ whose Beltrami coefficient is equal to the one of $F^{\mu}$ almost everywhere on the upper half plane $\mathbb{H}$ and vanishes on the lower half plane $\mathbb{L}.$ Then $F_{\mu}$ is conformal  at the escaping limit points $x\in\Lambda_{e}(G).$
\end{theorem}

\begin{proof}
As the statements of the theorem, we can choose a point $p_{0}\in S$ and a sufficiently large $R_{0}$ such that the support set of $f$ is contained in the disk  $B(p_{0}, R_{0}).$ Let $S_{R_{0}}=S\setminus B(p_{0}, R_{0})$ and let $\Omega_{R_{0}}$ be the lift of $S_{R_{0}}$ to the upper half plane $\mathbb{H}.$ By the definition of the escaping limit points, we know that an escaping geodesic eventually stays inside the region $\Omega_{R_{0}}$ and far from the support of $\mu.$

Since the {\mobi} transformations which keep the upper half plane invariant do not change the hyperbolic geometry properties(such as hyperbolic area of subset of $\mathbb{H}$ and hyperbolic distance between two points) of the upper half plane, with the conjugation of such {\mobi} transformations, we suppose $x=0$ and the initial point of the geodesic ray is $i$, denote the geodesic by $ \gamma(t)$, where $t$ is the arc-length parametrization with $\gamma(0)=i$ and $\lim_{t\rightarrow\infty}\gamma(t)=0.$ By the definition of escaping geodesic, there is a region such that none of the lifted pre-images of $B(p_{0}, R_{0})$ will hit the escaping geodesics eventually. Hence there is a sufficiently large $t_{0}$ $(t_{0}>1)$ and a $\delta\in (0,1)$, for $t>t_{0}$, dist$(\gamma(t), \mathbb{H}\setminus \Omega_{R_{0}})>\delta t>R_{0},$ where dist$(\cdot,\cdot)$ denote the hyperbolic distance between two points.

Let $r_{0}=e^{-t_{0}}$ and $\mu_{F}$ be the Beltrami coefficient of $F_{\mu}$. In the following, we will show that the integral $$\displaystyle\frac{1}{2\pi}\iint_{|z|<r_{0}}\frac{|\mu_{F}(z)|}{|z|^{2}}dxdy \eqno(3.2)$$
is finite.

Since the Beltrami coefficient  $\mu_{F}$ vanishes in the regions $\Omega_{R_{0}}$ and the lower half plane $\mathbb{L}$, we need to show that the integral (3.2) is finite in a neighborhood of $0$ outside the regions $\Omega_{R_{0}}$ and $\mathbb{L}.$ We will use polar coordinateto estimate the integral of (3.2). We have
$$\displaystyle\frac{1}{2\pi}\iint_{|z|<r_{0}}\frac{|\mu_{F}(z)|}{|z|^{2}}dxdy \leq \int^{r_{0}}_{0}dr\int^{\theta_{1}(r)}_{0}\displaystyle\frac{1}{r}d\theta
+ \int^{r_{0}}_{0}dr\int^{\theta_{2}(r)}_{0}\frac{1}{r}d\theta,\eqno(3.3)$$
where $\theta_{1}(r)$ and $\theta_{2}(r)$ are the arguments of the points which are the intersction of the hyperbolic circle
dist$(ir, z)=-\delta\ln r $ with the Euclidean circle $|z|=r$, where $r<1.$ Since the region is relative $\gamma (t)$ symmetry, we only need to show that
the integral  $$ \int^{r_{0}}_{0}dr\int^{\theta_{1}(r)}_{0}\displaystyle\frac{1}{r}d\theta \eqno(3.4)$$
 is finite.
By some easy calculation or see (\cite{Be}, Page 131), we know that the hyperbolic circle
dist$(ir, z)=-\delta\ln r $ is just the Euclidean circle

$$|z-ir\displaystyle\frac{r^{\delta}+r^{-\delta}}{2}|=r(\displaystyle\frac{r^{-\delta}-r^{\delta}}{2}).\eqno(3.4)$$
 Let $Q=x+iy$ be the intersection points of the hyperbolic circle
dist$(ir, z)=-\delta\ln r $ with the Euclidean circle $|z|=r$ in the first quadrant. Combine (3.5), we have
the imaginary part of $Q$ satisfies the equation $$y=\displaystyle\frac{2r}{r^{\delta}+r^{-\delta}}.$$
Therefore

 \begin{equation*}
\begin{aligned}
\sin \theta_{1}(r)=\displaystyle\frac{y}{r}=\displaystyle\frac{2}{r^{\delta}+r^{-\delta}}
\leq r^{\delta}.
\end{aligned}
\end{equation*}
Since for $\theta\in (0, \displaystyle\frac{\pi}{2})$, $\displaystyle\frac{2}{\pi}\theta \leq sin\theta,$ we have
$$\theta_{1}(r)\leq \displaystyle\frac{2}{\pi}r^{\delta}. \eqno (3.5)$$
Hence the integral   $$\int^{r_{0}}_{0}dr\int^{\theta_{1}(r)}_{0}\displaystyle\frac{1}{r}d\theta\leq \int^{r_{0}}_{0}r^{1-\delta}dr$$
is finite. Further more we have
$$\displaystyle\frac{1}{2\pi}\iint_{|z|<r_{0}}\frac{|\mu_{F}(z)|}{|z|^{2}}dxdy <\infty.\eqno(3.6)$$ By Lemma \ref{le1}, we know that $F_{\mu}$ is conformal at $0$.

Let \begin{equation*}
F(z)=\left\{
\begin{aligned}
F^{\mu}(z) &,& z\in \mathbb{H}, \\
\overline{F^{\mu}(\bar{z})} & ,& z\in \mathbb{C}\setminus\mathbb{H}.
\end{aligned}
\right.
\end{equation*}
By the symmetry of $F$ and (3.6), we have that
the quasi-symmetric homeomorphism $F^{\mu}|\mathbb{R}$ are differentiable at the points $x\in\Lambda_{e}(G)$.
\end{proof}

 For a compact quasi-conformal deformation as the statement of Theorem 3.1, in \cite{BJ}, Bishop and Jones use the properties of Schwarzian derivative of $F_{\mu}$ to estimate the Hausdorff dimension of $\Lambda_{e}(G)$ and $F_{\mu}(\Lambda_{e}(G))$, they showed
 \begin{theorem}\cite{BJ}
 As the statements of Theorem \ref{diff},  dim$\Lambda_{e}(G)$ is equal to $F_{\mu}(\Lambda_{e}(G))$, where dim$(\cdot)$ denotes the Hausdorff dimension of the set.
 \end{theorem}
As an application of Theorem \ref{diff}, we give a new proof of Bishop and Jones result.

\begin{proof}

By Theorem \ref{diff}  we know
$$ \Lambda_{e}(G)=\{x: F_{\mu}'(x)\text{  exists and non-zero, } x\in\Lambda_{e}(G).\}$$

Define the set  $$ \Lambda_{n}=\{x: \frac{1}{n}\leq |F_{\mu}'(x)|,  x\in\Lambda_{e}(G),\}$$
it is easy to see  $  \Lambda_{e}(G)=\cup_{n=1}^{\infty}\Lambda_{n}$ and $\Lambda_{n}\subset\Lambda_{n+1}.$

Hence $ \Lambda_{e}(G)=\underset{n\rightarrow\infty}\lim \Lambda_{n}.$

For  $x\in \Lambda_{n},$ we can choose a $\delta_{x}$ such that, for $|z-x|<\delta_{x}$,
$$\displaystyle\frac{1}{2n}\leq \displaystyle\frac{|F_{\mu}(z)-F_{\mu}(x)|}{|z-x|}\leq 2n.$$
This means that for each $x\in \Lambda_{n},$ there exists a constant $\delta_{x}$, such that for all  neighborhood $B_{x}$ of $x$ with $|B_{x}|<\delta_{x}$, $$\displaystyle\frac{1}{2n}|B_{x}|\leq F_{\mu}(|B_{x}|)\leq 2n|B_{x}|.\eqno (3.7)$$
Note that for fixed number $n$, the choice of constant $\delta_{x}$ depends on the points $x.$ To get rid of the dependence on $x$, define the set
$$ \Lambda_{n,k}=\{x:  x\in\Lambda_{n}, \forall |B_{x}| \text{ with } |B_{x}|<\displaystyle\frac{1}{k},~~ \displaystyle\frac{1}{2n}|B_{x}|\leq F_{\mu}(|B_{x}|)\leq 2n|B_{x}|.\}\eqno (3.8)$$
It is easy to see that $ \Lambda_{n,k}\subset  \Lambda_{n,k+1}$ and $\Lambda_{n}=\underset{k\rightarrow\infty}\lim \Lambda_{n,k}.$

In the following, we will show that for $\alpha\in (0,2),$
the Hausdorff measures of  $\mathcal{H}^{\alpha}(F_{\mu}(\Lambda_{n,k}))$ and $\mathcal{H}^{\alpha}(\Lambda_{n,k})$ satisfy
$$\displaystyle (\frac{1}{2n})^{\alpha}\mathcal{H}^{\alpha}(\Lambda_{n,k})\leq   \mathcal{H}^{\alpha}(F_{\mu}(\Lambda_{n,k})) \leq (2n)^{\alpha}\mathcal{H}^{\alpha}(\Lambda_{n,k}). \eqno (3.9)$$

We first show that the second inequality of $(3.9)$ holds.

For fixed $n$ and $k$, suppose $\{B_{i}\}$ is a cover of  $\Lambda_{n,k}$ with $|B_{i}|<\displaystyle \frac{1}{2nj}$, where $j\geq k.$
Then by the definition of $\Lambda_{n,k}$ we know that the sequence $\{F_{\mu}(B_{i})\}$ is a cover of $F_{\mu}(\Lambda_{n,k})$ with $$|F_{\mu}(B_{i})|<\displaystyle \frac{1}{j}.$$

For any  $\alpha\in (0,2)$, we have $$\mathcal{H}^{\alpha}_{1/j}(F_{\mu}(\Lambda_{n,k}))\leq \sum_{i=1}^{\infty} |F_{\mu}(B_{i})|^{\alpha}\leq \sum_{i=1}^{\infty}(2n|B_{i}|)^{\alpha}.\eqno(3.10)$$
Take the infimum of the right of $(3.9)$, we obtain
$$\mathcal{H}^{\alpha}_{1/j}(F_{\mu}(\Lambda_{n,k}))\leq(2n)^{\alpha}\mathcal{H}^{\alpha}_{1/2nj}(\Lambda_{n,k}).$$
Let $j$ tend to infinity, the $\alpha$- dimensional Hausdorff measures of  $F_{\mu}(\Lambda_{n,k})$ and $\Lambda_{n,k}$ satisfies
$$\mathcal{H}^{\alpha}(F_{\mu}(\Lambda_{n,k}))\leq(2n)^{\alpha}\mathcal{H}^{\alpha}(\Lambda_{n,k}).\eqno(3.11)$$
By (3.8) and the similar discussion as above we can get that the first inequality of (3.9) also holds.
By the definition of Hausdorff dimension, the inequalities (3.9) show that, for fixed $n, k$, the Hausdorff dimension of $\Lambda_{n,k}$ is the same as its image under the map $F_{\mu}.$ Since the dimension is preserve for every $n$ and $k$, hence we have
$$\text{dim }F_{\mu}(\Lambda_{e})=\text{dim}(\Lambda_{e}).$$ This completes the proof of this theorem.
\end{proof}

Now it is time to give the proof of Theorem\ref{main1}.

\section{Proof of Theorem \ref{main1} }
Let $G$ be a non-lattice  divergence Fuchsian group and $f$ be a quasi-conformal mapping of the surface $S=\mathbb{H}/G$ whose Beltrami coefficients is  compactly supported  on $S$.  As the statements of Theorem \ref{diff}, let $F_{\mu}$ be a quasi-conformal of the complex plane $\mathbb{C}$ which has  the same  Beltrami coefficient with the lifted mapping $F^{\mu}$ of $f$ to the upper half plane $\mathbb{H}$ and is conformal on the lower half plane $\mathbb{L}.$  By (\cite{BJ1}, Theorem 1.3), we know that the  $1$-dimensional Hausdorff  measure of $F_{\mu}(\Lambda_{e}(G))$ is zero.  Hence, as the notations in the proof of Theorem \ref{diff}, for fixed numbers $n$ and $k$, the Hausdorff measure of the subset $F_{\mu}(\Lambda_{n, k})$ is zero.  By (3.9) in the proof of Theorem \ref{diff}, we know, for fixed $n$ and $k$, the $1$-dimensional Hausdorff measure of $\Lambda_{n,k}$ is zero. Furthermore the  $1$-dimensional Hausdorff measure of $\Lambda_{e}(G)$ is zero.

In the following of this section we will show that the Hausdorff dimension of $\Lambda_{e}(G)$ is 1.

Since $G$ is non-lattice, the area of the surface $S$  and the generators of $G$ are both infinity.  The method we used here is from \cite{FM}. For the reader to better understand the distribution of the geodesics corresponding to $\Lambda_{e}(G)$ on the surface $S$, we give the detail of the proof here.

We first recall the definition of geodesic domain. A domain $D \subset S$ is called a geodesic domain if its relative boundary consists of finitely many non-intersecting closed simple geodesics and its area is finite. Fix a point $P_{0}\in S$, by (\cite{FM}, Theorem 4.1), we know that there exists a family $\{D_{i}\}_{i=0}^{\infty}$ of pairwise disjoint (except the boundary) geodesic domains  in $S$ satisfying that the boundary of $D_{i}$ and $D_{i+1}$ have at least a simple closed geodesic in common and $\lim_{i\rightarrow\infty}\text{ dist }(P_{0}, D_{i})=\infty,$ where dist$(\cdot, \cdot)$ denotes the hyperbolic distance of the surface $S.$

Let $\{D_{i}\}_{i=0}^{+\infty}$ be the family of geodesic domains of $S$ constructed as above.
 For any $i$, let $S_{i}$ be the Riemann surface obtained from $D_{i}$ by gluing a funnel along each one of the simple closed geodesics of its boundary. For each $i$, we choose a simple closed geodesic  $\gamma_{i}$ from the common boundary $D_{i}\cap D_{i+1}$ and a point $P_{i}\in \gamma_{i}.$ By (\cite{FM}, Theorem 4.1), we have
 $\delta_{i}\rightarrow 1$ when $i$ tends to infinity, where $\delta_{i}$ is the Poincare exponent of the surface $S_{i}.$

 For $\theta\in (0, \frac{1}{2}\pi)$, by (\cite{FM}, Theorem 5.1), we can choose a collection $\mathcal{B}_{i}$ of geodesics in $S_{i}$ with initial and final endpoint $P_{i}$ such that
$$L_{i}\leq\text{length}(\gamma)\leq L_{i}+C(P_{i}),\, \gamma\in \mathcal{B}_{i}£¬$$
where $L_{i}$ is a constant such that $L_{i}\rightarrow\infty$ as $i\rightarrow\infty$, $C(P_{i})$ is a constant depending only on the length of the geodesic $\gamma_{i}$, and $\sigma_{i}<\delta(S_{i})$, $\sigma_{i}\rightarrow 1$ as $i\rightarrow\infty$.

 The number of geodesic arcs in $B_{i}$ is at least $e^{L_{i}\sigma_{i}}$, and both the absolute value of the angles between $\gamma$ and the closed geodesic $\gamma_{i}$ are less than or equal to $\theta.$

 Note that for each $i$, $D_{i}$ is the convex core of $S_{i},$
implying that every geodesic arc $\gamma\in \mathcal{B}_{i}$ is contained in the convex core $D _{i}.$

  Furthermore, for each $i$, we may choose a geodesic arcs $\gamma_{i}^{*}$ with initial point $P_{i}$ and final endpoint $P_{i+1}$ such that
 $$L_{i}\leq\text{length}(\gamma^{*}_{i})\leq L_{i}+C(P_{i+1}),$$ and both the absolute value of the angles between $\gamma_{i}$, $\gamma_{i}^{*}$, and $\gamma_{i}^{*}$, $\gamma_{i+1}$  are less than or equal to $\theta.$

 In order to show the distribution of geodesics on $S$, we are going to construct a tree $\mathfrak{T}$ consisting of oriented geodesic arcs in the unit disk $\Delta.$

 Let us first lift $\gamma_{0}^{*}$ to the unit disk starting at $0$ (without loss of generality we may suppose that 0 projects onto $P_{0}$).  From the endpoint of the lifted $\gamma_{0}^{*}$ (which project onto $P_{1}$), lift the family $\mathfrak{B}_{1}$; from each of the end points of these liftings (which still project onto $P_{1}$), lift again $\mathfrak{B}_{1}$. Keep lifting $\mathfrak{B}_{1}$ in this way a total of $M_{1}$ times.

 Next, from each one of the endpoints obtained in the process above, we lift $\gamma_{1}^{*}$, and from each one of the endpoints of the liftings of $\gamma_{1}^{*}$ (which project onto $P_{2}$), we lift the collection $\mathfrak{B}_{2}$ sucessively $M_{2}$ times as above. Continuously this process
 indefinitely we obtain a tree $\mathfrak{T}.$

 It is easy to see that $\mathfrak{T}$ contains uncountably many branches.  The tips of the branches of $\mathfrak{T}$ are contained in the escaping limit set $\Lambda_{e}(G)$ of the covering group of $S$. For suitably choosing the sequence $\{M_{i}\}$ of repetitions, the dimension of the rims of tree $\mathfrak{T}$ is 1. By the construction of the tree $\mathfrak{T}$, we see that the tree $\mathfrak{T}$ is a unilaterally connected graph. Hence the geodesic corresponding to any branch of $\mathfrak{T}$ does not tend to the funnel with boundary $\gamma.$ Hence the dimension of the escaping limit set $\Lambda_{e}(G)$ of the covering group $G$ is 1.

\section{Proof of Theorem \ref{main2} }
 To prove this theorem, we need the following lemma which is essentially due to G\"{o}nye, see (\cite{Go}, P29.)

 \begin{lemma}\label{le2}
 Let $F$ be a quasi-symmetric homeomorphism of the real axis $\mathbb{R}$ and $A$ be a subset of $\mathbb{R}$ with Hausdorff dimension equal to $1$. If for any $x\in A$, $F'(x)$ exists and is non-zero, then the Hausdorff dimension of $F(A)$ is also $1.$
 \end{lemma}

 Now we give the proof of Theorem \ref{main2}.
 \begin{proof}
Let $G$ be a Fuchsian group of divergence type and not a lattice. Let $f$ be a quasi-conformal mapping on the surface  $\mathbb{H}/G.$ The lifting mapping $F^{\mu}$ of $f$ to the upper half plane $\mathbb{H}$ can extend to the real axis $\mathbb{R}$ naturally. We denote by $h=F^{\mu}|\mathbb{R}.$ The mapping $h$ is a quasi-symmetric homeomorphism of $\mathbb{R}.$ By Theorem \ref{diff}, the homeomorphism $h$ is differentiable at x in $\Lambda_{e}(G)$ with $|F^{\mu}(x)|\neq 0.$ By Theorem \ref{main1} and Lemma \ref{le2}, we know, for any $E\subset \mathbb{R},$
$$\text{max}(dim(E), dim f(\mathbb{R}\setminus E))=1.$$ Hence the theorem holds.
\end{proof}

\section{Proof of Theorem \ref{main3} }

\begin{proof}
The necessity of the equivalence is from (\cite{BS}, Theorem 4).

For the  sufficient condition, by Theorem \ref{main2}, we only need to show the case when $G$ is a Fuchsian group of the convergence type.

If $G$ is a Fuchsian group of the second, the boundary of Dirichlet fundamental domain contains at least an arc (denoted by $\alpha^{*}$) in $\mathbb{R}$.  It is easy to see that the homeomorphism is smooth on $\alpha^{*}$. Hence the sufficient condition holds.
 If $G$ is a  Fuchsian group of convergence type and of the first kind, we need to show that the Hausdorff dimension of the escaping limit set $\Lambda_{e}(G)$ is 1, actually it has positive $1$-dimensional Hausdorff measure.

Suppose $\gamma$ be a  closed geodesic on the surface $S=\mathbb{H}/G$. Consider the liftings of the closed geodesic $\gamma$ in the upper half plane $\mathbb{H}$. It consists of a nested set  $\Sigma$ of hyperbolic lines: the one intersecting the Dirichlet fundamental domain cuts it in two parts and we may assume that the point $i$ belongs to a part that has infinite (hyperbolic) area. The hyperbolic lines in $\Sigma$ of the first generation define a two-by-two disjoint family $(I_j)$ of intervals of the real axis $\mathbb{R}$. Suppose$\underset{i=1}\cup I_j$ is equal to $\mathbb{R}$  except a zero Lebesgue measure set,
 then almost every geodesic issued from $i$ would visit $\gamma$ infinitely often, contradicting of (\cite{FM}, Theorem 1). Thus the set of geodesics from $i$ that never visit $\gamma$ has positive measure. It follows that  the escaping limit set of $S$ has positive Lebesgue measure.

 Hence if $F^{\mu}$ corresponding to a compact quasi-ocnformal deformation of $G$,  we always have, for any $E\subset \mathbb{R}$,
 $$\text{max}(dim(E), dim f(\mathbb{R}\setminus E))=1.$$
\end{proof}

\end{document}